\documentclass[reqno]{amsart}
\usepackage{amsmath,amsthm,amscd,amssymb,amsfonts, amsbsy}
\usepackage{latexsym}
\usepackage{color}
\usepackage{enumerate}
\usepackage{pxfonts}
\usepackage{marginnote}
\usepackage{todonotes}

\theoremstyle{plain}
\newtheorem{theorem}[equation]{Theorem}
\newtheorem{lemma}[equation]{Lemma}

\theoremstyle{definition}

\theoremstyle{remark}
\newtheorem{remark}[equation]{Remark}

\newcommand{\dv}{\operatorname{div}}

\newcommand{\dist}{\operatorname{dist}}
\newcommand{\diam}{\operatorname{diam}}

\newcommand{\tr}{\operatorname{tr}}

\numberwithin{equation}{section}

\newcommand{\bR}{\mathbb{R}}

\providecommand{\set}[1]{\{#1\}}

\providecommand{\abs}[1]{\lvert#1\rvert}
\providecommand{\Abs}[1]{\left\lvert#1\right\rvert}

\providecommand{\norm}[1]{\lVert#1\rVert}

\begin{document}
\title[Green's functions of nondivergent elliptic operators]
{Note on Green's functions of non-divergence elliptic operators with continuous coefficients}

\author[H. Dong]{Hongjie Dong}
\address[H. Dong]{Division of Applied Mathematics, Brown University,
182 George Street, Providence, RI 02912, United States of America}
\email{Hongjie\_Dong@brown.edu}
\thanks{H. Dong was partially supported by the Simons Foundation, grant no. 709545, a Simons fellowship, grant no. 007638, and the NSF under agreement DMS-2055244.}

\author[S. Kim]{Seick Kim}
\address[S. Kim]{Department of Mathematics, Yonsei University, 50 Yonsei-ro, Seodaemun-gu, Seoul 03722, Republic of Korea}
\email{kimseick@yonsei.ac.kr}
\thanks{S. Kim is partially supported by National Research Foundation of Korea (NRF) Grant No. NRF-2019R1A2C2002724 and No. NRF-2022R1A2C1003322.}

\author[S. Lee]{Sungjin Lee}
\address[S. Lee]{Department of Mathematics, Yonsei University, 50 Yonsei-ro, Seodaemun-gu, Seoul 03722, Republic of Korea}
\email{sungjinlee@yonsei.ac.kr}

\subjclass[2010]{Primary 35J08, 35B45 ; Secondary 35J47}

\keywords{Green's function; Non-divergence elliptic equation; Dini mean oscillation}

\begin{abstract}
We improve a result in Kim and Lee (Ann. Appl. Math. 37(2):111--130, 2021): showing that if the coefficients of an elliptic operator in non-divergence form are of Dini mean oscillation, then its Green's function has the same asymptotic behavior near the pole $x_0$ as that of the corresponding Green's function for the elliptic equation with constant coefficients frozen at $x_0$.
\end{abstract}
\maketitle

\section{Introduction and main results}
We consider an elliptic operator $L$ in non-divergence form
\begin{equation*}						
L u=  a^{ij}(x) D_{ij} u,
\end{equation*}
where the coefficients $\mathbf{A}:=(a^{ij})$ are symmetric and satisfy the ellipticity condition:
\begin{equation}					\label{ellipticity-nd}
\lambda \abs{\xi}^2 \le a^{ij}(x) \xi^i \xi^j \le \Lambda \abs{\xi}^2,\quad \forall x \in \Omega,
\end{equation}
where $\lambda$ and $\Lambda$ are positive constants and $\Omega$ is a domain in $\bR^n$ with $n \ge 3$.
Here and below, we use the usual summation convention over repeated indices.

It is well known that the Green's function $G(x,y)$ for an elliptic operator in divergence form has the pointwise bound
\begin{equation}				\label{bound1}
G(x,y) \le C \abs{x-y}^{2-n}	\qquad (x, y \in \Omega,\;\;x\neq y)
\end{equation}
even when the coefficients are merely measurable; see e.g. \cite{GW82, LSW, HK07}.
However, the Green's function for an elliptic operator in non-divergence form does not necessarily enjoy the pointwise bound \eqref{bound1} even in the case when the coefficients $\mathbf A$ are uniformly continuous; see \cite{Bauman84}.
Recently, it is shown in \cite{HK20} that if the coefficients $\mathbf{A}$ are of Dini mean oscillation and if the domain is bounded and has $C^{1,1}$ boundary, then the Green's function exists and satisfies the pointwise bound \eqref{bound1}. Let us recall the definition of Dini mean oscillation.
For $x\in \bR^n$ and $r>0$, we denote by $B(x,r)$ the open ball with radius $r$ centered at $x$, and write $\Omega(x,r):=\Omega \cap B(x,r)$.
We denote
\[
\omega_{\mathbf A}(r, x):= \fint_{\Omega(x,r)} \,\abs{\mathbf{A}(y)-\bar {\mathbf A}_{\Omega(x,r)}}\,dy, \quad \text{ where } \;\bar{\mathbf A}_{\Omega(x,r)} :=\fint_{\Omega(x,r)} \mathbf{A},
\]
and we write
\[
\omega_{\mathbf A}(r, D):= \sup_{x \in D} \omega_{\mathbf A}(r, x) \quad \text{and}\quad \omega_{\mathbf A}(r)=\omega_{\mathbf A}(r, \overline{\Omega}).
\]
We say that $\mathbf{A}$ is of Dini mean oscillation in $\Omega$ if $\omega_{\mathbf A}(r)$ satisfies the Dini's condition
\begin{equation}				\label{eq2116sat}
\int_0^1 \frac{\omega_{\mathbf A}(t)}t \,dt <+\infty.
\end{equation}
Before proceeding further, let us clarify the relationship between Dini continuity and Dini mean oscillation.
We say that $\mathbf A$ is Dini continuous at $x_0 \in \Omega$ if
\begin{equation}				\label{eq_rho}
\varrho_{\mathbf A} (r, x_0)= \sup \set{\abs{\mathbf A(x)-\mathbf A(x_0)}:  x \in \Omega, \;\abs{x-x_0} \le r}
\end{equation}
satisfies the Dini's condition
\[
\int_0^1 \frac{\varrho_{\mathbf A}(t, x_0)}t \,dt <+\infty
\]
and that $\mathbf A$ is uniformly Dini continuous in $\Omega$ if
\[
\varrho_{\mathbf A}(r):= \sup_{x \in \Omega} \varrho_{\mathbf A}(r, x) \quad \text{satisfies}\quad \int_0^1 \frac{\varrho_{\mathbf A}(t)}t \,dt <+\infty.
\]
It is clear that if $\mathbf{A}$ is uniformly Dini continuous in $\Omega$, then $\mathbf{A}$ is of Dini mean oscillation in $\Omega$.
If $\mathbf{A}$ is of Dini mean oscillation in $\Omega$, then there is a modification   $\bar{\mathbf A}$ of $\mathbf A$ (i.e.,  $\bar{\mathbf A} =  \mathbf A$ a.e.) such that $\bar{\mathbf A}$ is uniformly continuous in $\Omega$ with its modulus of continuity controlled by $\omega_\mathbf{A}$. See \cite[Appendix]{HK20} for a proof.
However, a function of Dini mean oscillation is not necessarily Dini continuous; see \cite{DK17} for an example.

In a recent article \cite{KL21}, we gave an alternative proof of the above mentioned result in \cite{HK20} and established a sharp comparison with the Green's function for a constant coefficient operator.
In particular, we showed that
\begin{equation}				\label{asymptotics}
G(x_0,x)-G_{x_0}(x_0, x)= o(\abs{x-x_0}^{2-n})\;\text{ as }\; x\to x_0,
\end{equation}
where $G_{x_0}$ is the Green's function of the constant coefficient operator $L_{x_0}$ given by
\begin{equation}				\label{eq2203sun}
L_{x_0}u:= a^{ij}(x_0) D_{ij} u,
\end{equation}
provided that the mean oscillation of $\mathbf A$ satisfies so-called ``double Dini condition'' near $x_0$, that is, we have
\begin{equation}				\label{eq:ddmo}
\int_0^1 \frac{1}{s} \int_0^s \frac{\omega_{\mathbf A}(t,\Omega(x_0, r_0))}{t}\,dt\,ds = \int_0^1 \frac{\omega_{\mathbf A}(t,\Omega(x_0, r_0))\ln \frac{1}{t}}{t}\,dt <+\infty
\end{equation}
for some $r_0>0$.
It should be noted that the argument in \cite[Appendix]{HK20} reveals that if $\mathbf A$ satisfies the double Dini mean oscillation condition \eqref{eq:ddmo}, then there exists a modification of $\mathbf A$ that is Dini continuous in a neighborhood of $x_0$.
The asymptotic behavior \eqref{asymptotics} is well known for the Green's functions for elliptic operators in divergence form with continuous coefficients; see \cite{DM95}.
However, in the non-divergence form setting, this is a new result and it is one of the novelties in \cite{KL21}.

The aim of this note is to show that one can replace the condition \eqref{eq:ddmo} with a simple condition that $\mathbf A$ is continuous at $x_0$, that is,
\begin{equation}				\label{eq2114sat}
\lim_{r\to 0} \varrho_{\mathbf A}(r, x_0)=0.
\end{equation}
Notice that the coefficient matrix $\mathbf A=(a^{ij})$, which is assumed to be of Dini mean oscillation in $\Omega$, has a modification that is uniformly continuous in $\Omega$.
Therefore, without loss of generality, we shall hereafter assume that $\mathbf A$ is uniformly continuous in $\Omega$ so that \eqref{eq2114sat} holds for all $x_0 \in \Omega$.

We now state our main theorems.

\begin{theorem}					\label{thm01}
Let $\Omega$ be a bounded $C^{1,1}$ domain in $\bR^n$ with $n\ge 3$.
Assume the coefficients $\mathbf{A}=(a^{ij})$ of the operator $L$ satisfy the  condition \eqref{ellipticity-nd} and are of Dini mean oscillation in $\Omega$.
Then there exists the Green's function $G(x,y)$ of the operator $L$ in $\Omega$  and it satisfies the pointwise bound \eqref{bound1}.
Moreover, for any $x_0 \in \Omega$, we have
\begin{equation}				\label{eq1725sun}
\lim_{x\to x_0}\, \abs{x-x_0}^{n-2}\,\abs{G(x_0, x)-G_{x_0}(x_0, x)}=0,
\end{equation}
where $G_{x_0}$ is the Green's function of the constant coefficient operator $L_{x_0}$ as in \eqref{eq2203sun}.
\end{theorem}

\begin{theorem}					\label{thm02}
Under the same hypothesis of Theorem~\ref{thm01}, we also have
\begin{gather}				\label{eq1726sun}
\lim_{x\to x_0}\, \abs{x-x_0}^{n-2}\,\abs{G(x, x_0)-G_{x_0}(x, x_0)}=0,\\
						\label{eq1533th}
\lim_{x\to x_0}\, \abs{x-x_0}^{n-1}\,\abs{D_xG(x, x_0)-D_xG_{x_0}(x, x_0)}=0,\\
						\label{eq1534th}
\lim_{x\to x_0}\, \abs{x-x_0}^{n}\,\abs{D_x^2 G(x, x_0)-D_x^2G_{x_0}(x, x_0)}=0,
\end{gather}
for any $x_0 \in \Omega$.
\end{theorem}

\begin{remark}
As a matter of fact, the proof of Theorem~\ref{thm01} shall reveal that we have
\[
\lim_{x\to x_0}\, \abs{x-x_0}^{n-2} \,\abs{G^*(x,x_0)-G_{x_0}^\ast(x,x_0)} =0,
\]
where $G^\ast$ and $G_{x_0}^\ast$ are the Green's functions for the adjoint operators $L^\ast$ and $L_{x_0}^\ast$, respectively.
Therefore, \eqref{eq1726sun} can be regarded as the adjoint version of \eqref{eq1725sun}, and vice versa.
\end{remark}

\begin{remark}
We note that the estimates \eqref{eq1725sun} and \eqref{eq1726sun} are reminiscent of \cite[Eq.~(23)]{DM95}.
However, the method used in \cite{DM95} is not applicable to our setting because a local $L^\infty$ estimates corresponding to \cite[Lemma~4]{DM95} is not available  via $L^p$ theory for adjoint solutions of $L^\ast u =\dv^2 \mathbf{f}$.
Nevertheless, under the stronger condition that $\mathbf{A}\in C^\beta$ for some $\beta\in (0,1)$, similar to \cite[Theorem~2]{DM95}, we obtain
\begin{align*}
\abs{G(x_0, x)-G_{x_0}(x_0, x)} &=O\left(\abs{x-x_0}^{2-n+\beta}\right)\quad\text{as}\quad x \to x_0,\\
\abs{D^k_x G(x, x_0)-D^k_x G_{x_0}(x, x_0)} &=O\left(\abs{x-x_0}^{2-n-k+\beta}\right)\quad\text{as}\quad x \to x_0.\quad (k=0,1,2).
\end{align*}
Indeed, by modifying the proof of \cite[Theorem~2]{DM95}  and using the global Lorentz type estimate \cite[Lemmas 1 and 2]{DM95} and interior Schauder estimates for the solutions of $L^\ast u = \dv^2 \mathbf{f}$ and $Lu=f$, one can get the above estimates.
\end{remark}

A few further remarks are in order.
We point out that the proofs for Theorems \ref{thm01} and \ref{thm02} are readily extendable to elliptic systems, but they do not work in the two dimensional setting. See \cite{DK21} for the results regarding the two dimensional case as well as for a discussion on how to extend them to elliptic systems.
In \cite{KL21}, we assume that $\Omega$ is a bounded $C^{2,\alpha}$ domain.
This assumption was adopted there to obtain a global pointwise estimate for $D_{x}^2 G(x,y)$ only.
To prove that the pointwise bound \eqref{bound1} holds globally in $\Omega$, we just need a local $L^\infty$ estimate (valid up to the boundary) for solutions of the adjoint equation
\[
L^\ast u := D_{ij}(a^{ij}u)=0,
\]
which has been established earlier in \cite{DK17, DEK18}.
We particularly refer to \cite[Theorem~1.8]{DEK18}, where $\Omega$ is assumed to be a bounded $C^{1,1}$ domain.
However, in the scalar case, it is not even necessary to assume that $\Omega$ is a bounded $C^{1,1}$ domain for the global pointwise bound \eqref{bound1} to hold.
As it will be indicated in the proof of Theorem~\ref{thm01}, in the scalar case, it is enough to assume that $\Omega$ is a bounded $C^{1,\alpha}$ domain with $\alpha>\frac{n-1}{n}$ in order to utilize the $L^p$ solvability of the problem \eqref{adj_eq} for some $p>n/2$.
Also, since  we deal with the asymptotic estimates \eqref{eq1725sun} and \eqref{eq1726sun}, which are  interior estimates in nature, we only need an interior $L^\infty$ estimate for adjoint solutions to establish these estimates.

Finally, we should mention that there are many interesting papers dealing with the Green's functions or the fundamental solutions of non-divergence form elliptic and parabolic operators. See, for example,  \cite{Bauman84b, Bauman85, FS84, FGMS88, Krylov92, Esc2000, Cho11, MMc2010} and references therein.

\section{Preliminary lemmas}
In this section, we present some lemmas which will be used in the proof of Theorems \ref{thm01} and \ref{thm02}.
We need to consider the boundary value problem
\begin{equation}				\label{adj_eq}
L^\ast v = \dv^2 \mathbf{g} + f\;\text{ in }\;\Omega,\quad v=\frac{\mathbf{g} \nu\cdot \nu}{\mathbf{A}\nu\cdot \nu}\;\text{ on }\;\partial \Omega,
\end{equation}
where $\mathbf{g}=(g^{ij})$ is an $n \times n$ matrix-valued function,
\[
\dv^2 \mathbf{g}:=D_{ij}g^{ij},
\]
and $\nu$ is the unit exterior normal vector of $\partial\Omega$.
For $\mathbf{g} \in L^p(\Omega)$ and $f \in L^p(\Omega)$, where $1<p<\infty$ and $\frac{1}{p}+ \frac{1}{p'}=1$, we say that $v$ in $L^p(\Omega)$ is an adjoint solution of \eqref{adj_eq} if $v$ satisfies
\[
\int_\Omega v Lu = \int_\Omega \tr(\mathbf{g} D^2u) + \int_\Omega f u
\]
for any $u$ in $W^{2,p'}(\Omega) \cap W^{1,p'}_0(\Omega)$.

\begin{lemma}				\label{lem01}
Let $\Omega \subset \bR^n$ be a bounded $C^{1,1}$ domain.
Let $1<p<\infty$ and assume that $\mathbf{g} \in L^p(\Omega)$ and $f \in L^p(\Omega)$.
Then there exists a unique adjoint solution $u$ in $L^p(\Omega)$.
Moreover, we have the estimate
\[
\norm{u}_{L^p(\Omega)} \le C \left( \norm{\mathbf g}_{L^p(\Omega)} + \norm{f}_{L^p(\Omega)} \right),
\]
where the constant $C$ depends on $\Omega$, $p$, $n$, $\lambda$, $\Lambda$, and $\omega_{\mathbf A}$.
\end{lemma}
\begin{proof}
See  \cite[Lemma~2]{EM2016}.
\end{proof}

\begin{lemma}			\label{lem02}
Let $R_0>0$ and $\mathbf{g}=(g^{ij})$ be of Dini mean oscillation in $B(x_0, R_0)$.
Suppose $u$ is an $L^2$ solution of
\[
L^\ast u= \dv^2 \mathbf{g}\;\text{ in }\;B(x_0, 2r),
\]
where $0<r \le \frac12 R_0$.
Then we have
\[
\norm{u}_{L^\infty(B(x_0,r))} \le C \left(\fint_{B(x_0,2r)} \abs{u} + \int_0^{r} \frac{\omega_{\mathbf g}(t, B(x_0,2r))}{t}\,dt \right),
\]
where $C=C(n,\lambda, \Lambda, \omega_{\mathbf A}, R_0)$.
\end{lemma}
\begin{proof}
See \cite[Lemma~2.2]{KL21}.
\end{proof}

\section{Proof of Theorem~\ref{thm01}}
We slightly modify the argument in \cite[\S3.5]{KL21}.
Let us denote
\[
\mathbf{A}_{x_0}=\mathbf{A}(x_0), \quad L_{x_0} u := a^{ij}(x_0) D_{ij}u= \tr(\mathbf{A}_{x_0} D^2 u),
\]
and let $G_{x_0}(x,y)$ be the Green's function for $L_{x_0}$.
Since $L_{x_0}$ is an elliptic operator with constant coefficients, the existence of $G_{x_0}$ as well as the following pointwise bound is well known:
\begin{equation}				\label{eq1224th}
\abs{G_{x_0}(x,y)} \le  C\abs{x-y}^{2-n}\quad \quad (x\neq y),
\end{equation}
where $C=C(n, \lambda, \Lambda)$.
Moreover, since $\mathbf{A}_{x_0}$ is constant and symmetric, we have $L_{x_0}=L_{x_0}^\ast$ and $G_{x_0}$ is also symmetric, i.e.,
\begin{equation}				\label{eq1350m}
G_{x_0}(x,y)=G_{x_0}^\ast(y,x)=G_{x_0}(y,x) \qquad (x\neq y).
\end{equation}

Let us set
\[
\mathbf{g}=-(\mathbf{A}- \mathbf{A}_{x_0}) G_{x_0}(\cdot, x_0)
\]
and consider the problem
\begin{equation}				\label{eq1042th}
L^\ast v = \dv^2 \mathbf{g}\;\text{ in }\; \Omega,\quad
v=\frac{\mathbf{g}\nu\cdot \nu}{\mathbf{A}\nu\cdot \nu} \;\text{ on }\;\partial \Omega.
\end{equation}
Notice that the boundary condition in \eqref{eq1042th} simply reads that $v=0$ in $\partial\Omega$ since the Green's function $G_{x_0}(\cdot, x_0)$ vanishes on $\partial\Omega$.

By using that $\mathbf A$ is of Dini mean oscillation in $\Omega$, it is easy to verify that $\mathbf{g} \in L^q(\Omega)$ for  $q \in(1, \frac{n}{n-2})$; see \cite[Lemma~3.1]{KL21}.
Therefore, by Lemma~\ref{lem01}, there is a unique $v$ that belongs to $L^q(\Omega)$ for $q \in (1,\frac{n}{n-2})$.
Then, by the same argument as in \cite[\S3.1]{KL21}, we find that
\begin{equation}				\label{eq1531m}
G^\ast(\cdot, x_0):=G_{x_0}(\cdot, x_0) + v
\end{equation}
becomes the Green's function of $L^\ast$ in $\Omega$ with a pole at $x_0$.
Also, by following the same argument as in \cite[\S3.4]{KL21}, we find that the function $G(x,y)$ given by
\[
G(x,y)=G^\ast(y,x)\qquad (x\neq y)
\]
is the Green's function for $L$ in $\Omega$.
Note that we only need the assumption that $\Omega$ is a bounded $C^{1,1}$ domain to get these conclusions.
In fact, in the scalar case, we can even relax this assumption further to $\Omega$ being a $C^{1,\alpha}$ domain for some $\alpha >(n-1)/n$. See \cite[Appendix]{DK21}.

Now, for $y_0 \in \Omega$ with $y_0 \neq x_0$, let
\begin{equation}				\label{eq0043sat}
r=\tfrac15 \abs{y_0-x_0}\quad\text{and}\quad\delta=\min(\kappa r, \diam \Omega),
\end{equation}
where $\kappa>0$ is to be determined.
We assume that $r< \frac17 \mathrm{dist}(x_0,\partial\Omega)$ is so that $B(y_0, 2r)\subset \Omega$.
Let  $\zeta$ be a smooth function on $\bR^n$ such that
\[
0\le \zeta \leq 1, \quad \zeta=0 \;\text{ in }\;B(x_0,\delta/2),\quad \zeta=1\;\text{ in }\bR^n\setminus B(x_0,\delta),\quad \abs{D \zeta} \le 4/\delta.
\]
We then define $\mathbf{g}_1$ and $\mathbf{g}_2$ by
\begin{equation*}
\mathbf{g}_1= -\zeta(\mathbf{A}- \mathbf{A}_{x_0}) G_{x_0}(\cdot, x_0)\quad\text{and}\quad \mathbf{g}_2= -(1-\zeta)(\mathbf{A}-\mathbf{A}_{x_0}) G_{x_0}(\cdot, x_0).
\end{equation*}
Recall the definition \eqref{eq_rho} and note that $\norm{\mathbf A-\mathbf A_{x_0}}_\infty \le C(n, \Lambda)$.

Choose $p_1 \in (\frac{n}{n-2},\infty)$ and $p_2 \in (1, \frac{n}{n-2})$; e.g., take $p_1=\frac{2n}{n-2}$ and $p_2=\frac{n-1}{n-2}$.
By \eqref{eq1224th} and properties of $\zeta$, we then have
\begin{align}
					\label{eq1730f}
\int_\Omega \abs{\mathbf{g}_1}^{p_1} &\le C  \int_{\Omega\setminus \Omega(x_0,\delta/2)} \abs{x-x_0}^{(2-n)p_1}\,dx \le C\delta^{(2-n)p_1+n},\\
						\label{eq1432f}
\int_\Omega\, \abs{\mathbf{g}_2}^{p_2}& \le C \varrho_{\mathbf A}(\delta,x_0)^{p_2} \int_{\Omega(x_0,\delta)} \abs{x-x_0}^{(2-n)p_2}dx \le C \varrho_{\mathbf A}(\delta,x_0)^{p_2} \delta^{(2-n)p_2+n}.
\end{align}
Let $v_i$  be the solutions of the problems
\[
L^\ast v_i = \dv^2 \mathbf{g}_i\;\text{ in }\; \Omega,\quad
v_i=\frac{\mathbf{g}_i\nu\cdot \nu}{\mathbf{A}\nu\cdot \nu} \;\text{ on }\;\partial \Omega. \qquad(i=1,2)
\]
By Lemma~\ref{lem01} together with \eqref{eq1730f} and \eqref{eq1432f}, we have
\begin{equation}				\label{eq1618f}
\norm{v_1}_{L^{p_1}(\Omega)} \le C\delta^{2-n+\frac{n}{p_1}} \qquad\text{and}\qquad
\norm{v_2}_{L^{p_2}(\Omega)} \le C\varrho_{\mathbf A}(\delta,x_0) \delta^{2-n+\frac{n}{p_2}}.
\end{equation}
Since $\mathbf g_1$ and $\mathbf g_2$ also belong to $L^{q}(\Omega)$ for $q \in (1,\frac{n}{n-2})$, we have $v_1+v_2 \in L^q(\Omega)$ for $q\in (1,\frac{n}{n-2})$.
Therefore, we conclude that
\[
v=v_1+v_2
\]
by the uniqueness of solutions to the problem \eqref{eq1042th}.

Now, we estimate
$v_1(y_0)$ and $v_2(y_0)$ as follows.
By Lemma~\ref{lem02}, we have
\begin{equation}				\label{eq1335sat}
\abs{v_i(y_0)} \le  C \fint_{B(y_0,2r)} \abs{v_i} + C \int_0^{r} \frac{\omega_{\mathbf{g}_i}(t, B(y_0,2r))}{t}\,dt,\qquad (i=1,2).
\end{equation}
Using \eqref{eq1618f} together with H\"older's inequalities, we have
\begin{equation}				\label{eq1340sat}
\begin{aligned}
\fint_{B(y_0,2r)} \abs{v_1} &\le C r^{-\frac{n}{p_1}}  \norm{v_1}_{L^{p_1}(B(y_0, 2r))} \le C r^{-\frac{n}{p_1}}  \delta^{2-n+\frac{n}{p_1}},\\
\fint_{B(y_0,2r)} \abs{v_2} &\le Cr^{-\frac{n}{p_2}}  \norm{v_2}_{L^{p_2}(B(y_0, 2r))} \le C\varrho_{\mathbf A}(\delta,x_0) r^{-\frac{n}{p_2}}  \delta^{2-n+\frac{n}{p_2}}.
\end{aligned}
\end{equation}
The following technical lemma is a variant of \cite[Lemma~3.2]{KL21}.
In fact, the proof is much shorter under the assumption that  $\mathbf A$ is continuous at $x_0$.
\begin{lemma}			\label{lem03}
Let $\eta$ be a Lipschitz function on $\bR^n$ such that $0 \le \eta \le 1$ and $\abs{D \eta} \le 4/\delta$ for some $\delta>0$.
Let
\[
\mathbf{g}=-\eta (\mathbf{A}-\mathbf{A}_{x_0}) G_{x_0}(\cdot, x_0)
\]
and for $y_0 \in \Omega$ with $y_0 \neq x_0$, let  $r=\frac15 \abs{x_0-y_0}$.
Then, for any $t \in (0, r]$ we have
\[
\omega_{\mathbf{g}}(t, \Omega(y_0,2r)) \le C r^{2-n}\left( \omega_{\mathbf A}(t)+  \,\frac{r+\delta}{\delta r} \,\varrho_{\mathbf A}(8r,x_0) t \right),
\]
where $C=C(n, \lambda, \Lambda, \Omega)$.
\end{lemma}
\begin{proof}
For $\bar x \in \Omega(y_0, 2r)$ and $0<t \le r$, we have
\begin{align*}
\omega_{\mathbf{g}}(t,\bar x)&=\fint_{\Omega(\bar x,t)} \,\Abs{(\mathbf A-\mathbf A_{x_0}) G_{x_0}(\cdot, x_0) \eta- \overline{((\mathbf A-\mathbf A_{x_0}) G_{x_0}(\cdot, x_0) \eta)}_{\Omega(\bar x, t)}}\\
&\le \fint_{\Omega(\bar x,t)} \,\Abs{(\mathbf{A}-\mathbf A_{x_0}) G_{x_0}(\cdot, x_0) \eta- \overline{(\mathbf{A}-\mathbf{A}_{x_0})}_{\Omega(\bar x, t)} G_{x_0}(\cdot, x_0) \eta}\\
&\qquad \quad+ \fint_{\Omega(\bar x,t)} \,\Abs{\overline{(\mathbf{A}-\mathbf{A}_{x_0})}_{\Omega(\bar x, t)} G_{x_0}(\cdot, x_0) \eta- (\overline{(\mathbf A-\mathbf A_{x_0}) G_{x_0}(\cdot, x_0) \eta})_{\Omega(\bar x, t)}}\\
&=:I+II.
\end{align*}
Observe that we have $\dist(x_0, \Omega(\bar x, t)) \ge 2r$ and thus for $x$, $y \in \Omega(\bar x,t)$, by \eqref{eq1224th} and the interior gradient estimate for elliptic equations with constant coefficients, we have
\begin{equation}				\label{eq1225th}
\abs{G_{x_0}(x, x_0)- G_{x_0}(y, x_0)} \le C t r^{1-n},
\end{equation}
where $C=C(n,\lambda, \Lambda, \Omega)$.
Since $\dist(x_0, \Omega(\bar x, t)) \ge 2r$, by using \eqref{eq1224th} we obtain
\begin{align}				\nonumber
I &\le \fint_{\Omega(\bar x, t)} \,\Abs{(\mathbf{A}-\mathbf{A}_{x_0})- \overline{(\mathbf{A}-\mathbf{A}_{x_0})}_{\Omega(\bar x, t)}} \abs{G_{x_0}(\cdot, x_0)}\\
						\label{eq2222th}
&\le \fint_{\Omega(\bar x, t)} \,C r^{2-n} \abs{\mathbf{A}-\bar{\mathbf A}_{\Omega(\bar x,t)}} \le C r^{2-n} \omega_{\mathbf A}(t).
\end{align}
Also, we have
\begin{align}				\nonumber
II &\le \fint_{\Omega(\bar x, t)} \,\Abs{\fint_{\Omega(\bar x, t)}(\mathbf{A}(y)-\mathbf{A}(x_0)) \left(G_{x_0}(x, x_0)\eta(x)- G_{x_0}(y, x_0)\eta(y)\right) dy} dx \\
						\label{eq2223th}
&\le \fint_{\Omega(\bar x, t)} \fint_{\Omega(\bar x, t)} \abs{\mathbf{A}(y)-\mathbf{A}(x_0)}\, \abs{G_{x_0}(x, x_0)\eta(x)- G_{x_0}(y, x_0)\eta(y)}\,dy\,dx.
\end{align}
By using \eqref{eq1224th}, \eqref{eq1225th} and $\abs{D \eta} \le 4/\delta$, we have  for $x$, $y \in \Omega(\bar x, t)$ that
\begin{align}				\nonumber
&\abs{G_{x_0}(x, x_0)\eta(x)- G_{x_0}(y, x_0)\eta(y)}\\
						\nonumber
&\qquad\qquad  \le \abs{G_{x_0}(x, x_0)- G_{x_0}(y, x_0)}\,\abs{\eta(x)} + \abs{G_{x_0}(y, x_0)}\,\abs{\eta(x)- \eta(y)}\\
						\label{eq2224th}						
&\qquad\qquad  \le C t r^{1-n}+C r^{2-n} t/\delta \le C t r^{1-n}(1+r/\delta).
\end{align}
Plugging \eqref{eq2224th} into \eqref{eq2223th} and noting that for $y \in \Omega(\bar x,t)$, we have
\[
\abs{y-x_0} \le \abs{y-\bar x}+\abs{\bar x-y_0}+\abs{y_0-x_0} <t+2r+5r=8r,
\]
we obtain
\begin{equation}				\label{eq2226th}					
II \le C t r^{1-n}(1+r/\delta) \varrho_{\mathbf A}(8r,x_0).
\end{equation}
Combining \eqref{eq2222th} and \eqref{eq2226th}, we have (recall $t\le r$)
\begin{equation*}			
\omega_{\mathbf{g}}(t, \bar x) \le I+II \le C r^{2-n}\left( \omega_{\mathbf A}(t)+ t(1/r+1/\delta)\varrho_{\mathbf A}(8r,x_0) \right).
\end{equation*}
The lemma is proved by taking supremum over $\bar x \in \Omega(y_0, 2r)$.
\end{proof}

By applying Lemma~\ref{lem03} with $\eta=\zeta$ and $\eta=1-\zeta$, respectively, we have
\begin{equation}				\label{eq1331sat}
\int_0^{r} \frac{\omega_{\mathbf{g}_i}(t, B(y_0,2r))}{t}\,dt \le C r^{2-n} \left(\int_0^r \frac{\omega_{\mathbf A}(t)}{t}\,dt+\frac{\delta+r}{\delta}\varrho_{\mathbf A}(8r,x_0) \right)
\quad (i=1,2).
\end{equation}
Now we substitute \eqref{eq1340sat} and \eqref{eq1331sat} back to \eqref{eq1335sat} to obtain (recall $v=v_1+v_2$)
\begin{multline}				\label{eq1406sat}
\abs{v(y_0)}
\le C r^{2-n}\left( r^{n-2-\frac{n}{p_1}}\delta^{2-n+\frac{n}{p_1}}+ \varrho_{\mathbf A}(\delta,x_0)r^{n-2-\frac{n}{p_2}}\delta^{2-n+\frac{n}{p_2}} \right.\\
\left.+\int_0^r \frac{\omega_{\mathbf A}(t)}{t}\,dt + \frac{\delta+r}{\delta}\varrho_{\mathbf A}(8r,x_0) \right).
\end{multline}
For any $\epsilon >0$, we can choose $\kappa$ sufficiently large so that
$\kappa^{2-n+\frac{n}{p_1}}< \epsilon/2$.
Then choose $r_0$ sufficiently small so that $r_0<\min \left(\frac17 \mathrm{dist}(x_0,\partial\Omega), \frac{5}{\kappa}\diam \Omega \right)$ and
\[
\varrho_{\mathbf A}(\kappa r_0,x_0) \kappa^{2-n+\frac{n}{p_2}} + \int_0^{r_0} \frac{\omega_{\mathbf A}(t)}{t}\,dt + \frac{\kappa+1}{\kappa}\varrho_{\mathbf A}(8r_0,x_0) <\frac12 \epsilon,
\]
which is possible due to \eqref{eq2114sat} and \eqref{eq2116sat}.
Then, it follows from \eqref{eq1531m}, \eqref{eq1406sat}, and \eqref{eq0043sat} that if $\abs{y_0-x_0}< r_0$, then we have
\begin{equation}				\label{eq0022sat}
\abs{G^\ast(y_0,x_0)-G_{x_0}(y_0,x_0)} \le C \epsilon  \abs{y_0-x_0}^{2-n}.
\end{equation}
Therefore, we conclude that
\[
\lim_{x\to x_0}\, \abs{x-x_0}^{n-2} \,\abs{G^*(x,x_0)-G_{x_0}(x,x_0)} =0
\]
because $\epsilon>0$ is arbitrary and \eqref{eq0022sat} holds for any $y_0 \in \Omega$ satisfying $0<\abs{y_0-x_0}<r_0$.
Finally, we obtain \eqref{eq1725sun} from the above by noting \eqref{eq1350m} and the relation $G(x,y)=G^\ast(y,x)$ for $x\neq y$.
\qed

\section{Proof of Theorem~\ref{thm02}}
We modify the proof of Theorem~\ref{thm01} as follows.
Notice that for $x_0\neq y_0$ we have
\[
G(y_0,x_0)- G_{x_0}(y_0, x_0)= G^\ast(x_0, y_0)- G_{x_0}(x_0,y_0),
\]
which suggests that we should consider
\[
\tilde v=\tilde v(x)= G^\ast(x, y_0)- G_{x_0}(x,y_0).
\]
Note that $\tilde v$ satisfies
\[
L^\ast \tilde v = \dv^2 \tilde{\mathbf g}\;\text{ in }\; \Omega,\quad
v=\frac{\tilde{\mathbf g}\nu\cdot \nu}{\mathbf{A}\nu\cdot \nu} \;\text{ on }\;\partial \Omega,
\]
where
\[
\tilde{\mathbf g}=-(\mathbf{A}- \mathbf{A}_{x_0}) G_{x_0}(\cdot, y_0).
\]
Let $r$ and $\delta$ be the same as in \eqref{eq0043sat}, and let $\tilde \zeta$ be a smooth function on $\bR^n$ such that
\[
0\le \tilde \zeta \leq 1, \quad \tilde \zeta=0 \;\text{ in }\;B(y_0,\delta/2),\quad \tilde \zeta=1\;\text{ in }\bR^n\setminus B(y_0,\delta),\quad \abs{D \tilde \zeta} \le 4/\delta.
\]
We define $\tilde{\mathbf g}_1$ and $\tilde{\mathbf g}_2$ by
\begin{equation*}
\tilde{\mathbf g}_1= -\tilde\zeta(\mathbf A- \mathbf{A}_{x_0}) G_{x_0}(\cdot, y_0)\quad\text{and}\quad \tilde{\mathbf g}_2= -(1-\tilde\zeta)(\mathbf A-\mathbf{A}_{x_0}) G_{x_0}(\cdot, y_0).
\end{equation*}
Similar to \eqref{eq1730f} and \eqref{eq1432f}, we have
\begin{align*}
\int_\Omega \abs{\tilde{\mathbf g}_1}^{p_1} &\le C  \int_{\Omega\setminus \Omega(y_0,\delta/2)} \abs{x-y_0}^{(2-n)p_1}\,dx \le C\delta^{(2-n)p_1+n},\\
\int_\Omega\, \abs{\tilde{\mathbf g}_2}^{p_2} & \le C  \varrho_{\mathbf A}(\delta + 5r,x_0)^{p_2} \int_{\Omega(y_0,\delta)} \abs{x-y_0}^{(2-n)p_2}dx \le C \varrho_{\mathbf A}(\delta+5r,x_0)^{p_2} \delta^{(2-n)p_2+n}.
\end{align*}
We shall assume that $r<\frac14 \mathrm{dist}(x_0,\partial\Omega)$ is so that $B_{2r}(x_0)\subset \Omega$.
Then by essentially the same proof of Lemma~\ref{lem03}, we have
\[
\omega_{\tilde{\mathbf g}_i}(t, B(x_0,2r)) \le C r^{2-n}\left( \omega_{\mathbf A}(t)+  \,\frac{r+\delta}{\delta r} \,\varrho_{\mathbf A}(3r,x_0) t \right).
\]
Then, similar to \eqref{eq1406sat}, we get
\begin{multline*}
\abs{\tilde v(x_0)} \le C r^{2-n}\left(  r^{n-2-\frac{n}{p_1}}\delta^{2-n+\frac{n}{p_1}}+ \varrho_{\mathbf A}(\delta+5r,x_0) r^{n-2-\frac{n}{p_2}}\delta^{2-n+\frac{n}{p_2}} \right.\\
\left. +\int_0^r \frac{\omega_{\mathbf A}(t)}{t}\,dt + \frac{\delta+r}{\delta}\varrho_{\mathbf A}(3r,x_0) \right).
\end{multline*}
By using the last inequality in place of \eqref{eq1406sat} and proceeding similarly as in the proof of Theorem~\ref{thm01}, we find that, similar to \eqref{eq0022sat}, for any $\epsilon>0$, there exists a positive $r_0<\frac14 \mathrm{dist}(x_0,\partial\Omega)$ such that if $\abs{x_0-y_0}< r_0$, then we have
\[
\abs{G^\ast(x_0,y_0)-G_{x_0}(x_0,y_0)} \le C \epsilon  \abs{y_0-x_0}^{2-n}.
\]
Since $\epsilon>0$ is arbitrary and the last inequality is true for any $y_0 \in \Omega$ satisfying $0<\abs{y_0-x_0}<r_0$, we obtain \eqref{eq1726sun} from the last inequality by using  the relations $G(x,y)=G^\ast(y,x)$ and $G_{x_0}(x,y)=G_{x_0}(y,x)$ for $x\neq y$.

To establish \eqref{eq1533th}, we note that $v:=G(\cdot, x_0)-G_{x_0}(\cdot, x_0)$ satisfies
\[
L v=-(\mathbf A- \mathbf A_{x_0})D^2 G_{x_0}(\cdot, x_0)=:f.
\]
For any $y_0 \in \Omega$ satisfying $0<\abs{y_0-x_0}<\frac14 \dist(x_0, \partial\Omega)$, let $r=\frac15 \abs{y_0-x_0}$.
Since  $G_{x_0}$ is the Green's function of constant coefficient operator $L_{x_0}$ and $B(x_0,20r) \subset \Omega$, we have
\[
\abs{D^k_x G_{x_0}(x, x_0)} \le C \abs{x-x_0}^{2-n-k},\quad \forall x \in B(x_0, 10r) \setminus \set{x_0},\;\; k=1,2,\ldots.
\]
Then we have
\begin{equation}					\label{eq1944th}
\norm{f}_{L^\infty(B(y_0, 3r))} \le C \varrho_{\mathbf A}(8r, x_0) r^{-n}.
\end{equation}
Also, by a similar computation as in the proof of Lemma~\ref{lem03}, we have
\begin{equation}					\label{eq1945th}
\omega_{f}(t, B(y_0,2r)) \le C r^{-n}\left( \omega_{\mathbf A}(t)+  \,\frac{t}{r} \,\varrho_{\mathbf A}(8r,x_0) \right).
\end{equation}

Therefore, by using \eqref{eq1726sun}, \eqref{eq1944th}, the local $W^{2,p}$ estimate
\begin{equation}					\label{eq1946th}
r^2 \norm{D^2 v}_{L^p(B_{2r}(y_0))}+ r \norm{D v}_{L^p(B_{2r}(y_0))}
\le C \norm{v}_{L^p(B_{3r}(y_0))} +Cr^2 \norm{f}_{L^p(B_{3r}(y_0))},
\end{equation}
and the Sobolev embedding, we have
\begin{equation}				\label{eq2248th}
r^{n-1} \abs{Dv(y_0)} =o(r),
\end{equation}
where we use $o(r)$ to denote some bounded quantity that tends to $0$ as $r\to 0$.

Moreover, similar to Lemma~\ref{lem02}, the proof of \cite[Theorem~1.6]{DK17} reveals that
\[
\norm{D^2 v}_{L^\infty(B(y_0,r))} \le C \left(\fint_{B(y_0,2r)} \abs{D^2 v} + \int_0^{r} \frac{\omega_{f}(t, B(y_0,2r))}{t}\,dt \right).
\]
Therefore, by \eqref{eq1944th}, \eqref{eq1946th}, and \eqref{eq1945th}, we have
\begin{equation}					\label{eq2249th}
r^n \abs{D^2 v(y_0)} = o(r).
\end{equation}
Now, \eqref{eq1533th} and \eqref{eq1534th} follow from \eqref{eq2248th} and \eqref{eq2249th}, respectively, recalling that $v=G(\cdot, x_0)-G_{x_0}(\cdot, x_0)$ and $r=\frac15 \abs{x_0-y_0}$.
\qed


\end{document}